\documentclass[11pt,a4paper]{amsart}
\usepackage[T1]{fontenc}
\usepackage{lmodern}
\usepackage[margin=27mm]{geometry}
\usepackage{amsmath,amssymb,amsfonts,mathtools,mathrsfs}
\usepackage{microtype}
\usepackage{enumitem}
\usepackage{needspace}
\usepackage{xcolor}
\usepackage{hyperref} 
\hypersetup{colorlinks=true,linkcolor=blue!45!black,citecolor=blue!45!black,
  urlcolor=blue!55!black,pdftitle={A conjecture on the action of Hecke
operators}}
\setlist[enumerate]{leftmargin=*,itemsep=3pt}
\setlist[itemize]{leftmargin=*,itemsep=3pt}
\numberwithin{equation}{section}
\newtheorem{theorem}{Theorem}[section]
\newtheorem{claim}[theorem]{Claim}
\newtheorem{proposition}[theorem]{Proposition}
\newtheorem{lemma}[theorem]{Lemma}
\newtheorem{corollary}[theorem]{Corollary}
\newtheorem{definition}[theorem]{Definition} 
 \newcommand{\emp}{\emptyset}  \newcommand{\tr}{\operatorname{Tr}}

\newtheorem{conjecture}[theorem]{Conjecture}
\newtheorem{remark}[theorem]{Remark}

\newcommand{\OO}{\mathcal O}
\newcommand{\XX}{\widehat\Xi}
\DeclareMathOperator{\inv}{inv} \newcommand{\kk}{\kappa}

\newcommand{\C}{\mathbb C}
\newcommand{\Z}{\mathbb Z}

\newcommand{\PP}{\mathbb P}
\newcommand{\EE}{\mathcal E}
\newcommand{\KK}{\mathcal K}
\newcommand{\SSS}{\mathcal S}
\newcommand{\MM}{\mathcal M}
\newcommand{\HH}{\mathcal H}
\newcommand{\DD}{\mathcal D}

\newcommand{\mcM}{\mathcal M}
\newcommand{\BB}{\mathfrak B} 

\newcommand{\GG}{\mathbf G} 
\newcommand{\Bun}{\operatorname{Bun}}
\newcommand{\GL}{\operatorname{GL}}
\newcommand{\Nrd}{\operatorname{Nrd}}
\newcommand{\Trd}{\operatorname{Trd}}
\newcommand{\End}{\operatorname{End}}
\newcommand{\Hom}{\operatorname{Hom}}

\newcommand{\rk}{\operatorname{rk}}
\newcommand{\ind}{\operatorname{ind}}

\newcommand{\id}{\operatorname{id}}
\newcommand{\vol}{\operatorname{vol}}
\newcommand{\sst}{\mathrm{ss}}

\newcommand{\dens}{\mathscr D}

\newcommand{\sms}{\smallskip}

\DeclareMathOperator{\Nm}{N}
\title{A conjecture on the action of Hecke operators}
\author{David Kazhdan}
\begin{document}
\begin{abstract} Let $F$ be a non-Archimedean local field of characteristic different
from $2$, and let $E/F$ be a quadratic \'etale algebra. For central
division $F$-algebras $L$ of fixed degree $n$, the double quotients
$L^\times\backslash L_E^\times/L^\times$ are identified with a common
topological space $\XX_E$. Each $L$ gives a convolution product $*_E^L$
on the space $\MM_E$ of compactly supported measures on $\XX_E$.
The first conjecture asserts that this product is independent of $L$.
I prove it for split $E$ and, for every $E$, in degrees $1,2,3,4,6$.
The second part of the article defines geometric Hecke operators $T_D^L$
for finite \'etale divisors $D$ of $\PP^1_F$ disjoint from the level divisor,
and identifies their convolution measures $\mu_D^L\in\MM_E$.
The second conjecture asserts that these measures are independent of $L$.
\end{abstract}
\maketitle
\section{The setup}
\subsection{Convolution products}

\begin{definition}
For a locally compact Hausdorff space $Z$, let $\mcM(Z)$ denote the
space of compactly supported complex Radon measures. I give it the
locally convex inductive-limit topology of the spaces $\mcM_C(Z)$,
where $C$ runs over the compact subsets of $Z$ and $\mcM_C(Z)$
has the weak topology $\sigma(\mcM_C(Z),C(C))$.
\end{definition}
This topology is Hausdorff: integration against every function in
$C_c(Z)$ is continuous and these functions separate Radon measures. The convolution below is separately continuous and need not be jointly
continuous.

\sms
Let $G$ be a locally compact unimodular group, $K\subset G$ a compact
subgroup and $dk$ the Haar probability measure on $K$. Let $ Y=G/K,\qquad X=K\backslash G/K $ where both quotients carry the quotient topologies and  
$ q:G \longrightarrow X $ and $p: G  \longrightarrow Y$ be  the quotient maps.
The group $G$ acts on
$\mcM(Y)$ by left translations.
\begin{definition}
\begin{enumerate}
\item $\mcM^{K\times K}(G)$ is the algebra of compactly supported
bi-$K$-invariant measures, with convolution
$\mu_1*\mu_2=m_*(\mu_1\boxtimes\mu_2)$ where $m: G\times G \to G$ is the product.
\item $\End_G(\mcM(Y)) $ is the algebra of continuous complex-linear
$G$-equivariant endomorphisms of $\mcM(Y)$, with multiplication
by composition.
\item For $\mu\in\mcM^{K\times K}(G)$, define the Hecke operator $T_\mu \in \End_G(\mcM(Y)) $ by
\[
T_\mu(\lambda)=\widetilde\lambda*\mu\quad\text{pushed forward to }Y,
\]
where $\widetilde\lambda$ is the unique right-$K$-invariant measure
on $G$ with $p_*\widetilde\lambda=\lambda$. Equivalently,
$T_\mu(\delta_{gK})=g_*p_*\mu$.

\item For  $\nu\in\mcM(Y), f\in C_c(G)$ we define the right-$K$-invariant lift $\widetilde\nu$ of $\nu$ by
\[
 \int_G f\,d \widetilde\nu
 =\int_Y\left(\int_K f(gk)\,dk\right)d\nu(gK).
\]

\end{enumerate}
\end{definition}

\begin{lemma}\label{isom}
The map
\[
 i:\mcM^{K\times K}(G)\longrightarrow\End_G(\mcM(Y)),
 \qquad \mu\longmapsto T_\mu,
\]
is an anti-isomorphism of algebras.
\end{lemma}

\begin{proof} 
\begin{claim}
\begin{enumerate}
\item The measure $\widetilde\nu$ on $Y$ is compactly supported. 
\item 
The spaces $X $ and $Y$ are Hausdorff and 
$p$ and $q$ are open proper quotient maps.
\item $q_\ast$ defines an isomorphism $\mcM (G) {K\times K}(G)\longrightarrow \mcM (X)$ of topological vector spaces whose inverse is given by the averaging over the
 $K\times K$-orbits.
\item $p_\ast$ defines an isomorphism $\mcM (G) K(G)\longrightarrow \mcM (Y)$ of topological vector spaces whose inverse is given by the averaging over the
 $K$-orbits.
\end{enumerate}
\end{claim}

\sms

For  $A\in\End_G(\mcM(Y))$ and we define  $\nu _A=A(\delta_{y_0})$,
where $y_0=K$. Equivariance implies $\nu\in\mcM(Y)^K$.
Let $\mu$ be its unique bi-$K$-invariant lift to $G$. For every $g\in G$,
\[
A(\delta_{gK})=g_*\nu=T_\mu(\delta_{gK}).
\]
Finite linear combinations of Dirac measures supported in a fixed
compact subset $C$ are dense in $\mcM_C(Y)$ for its weak topology.
Since $A$ and $T_\mu$ are continuous, they agree on $\mcM_C(Y)$
for every compact $C$, and hence on $\mcM(Y)$. This proves bijectivity.
Finally, associativity of convolution gives
$T_{\mu_2}(T_{\mu_1}(\lambda))=T_{\mu_1*\mu_2}(\lambda)$.
\end{proof}

\begin{corollary}\label{iso}
$j$ is  linear bijection
$j:\mcM(X)\to\End_G(\mcM(Y))$ where  $j:=i\circ q_*^{-1}$.

\end{corollary}
\begin{definition}\label{product} The  algebra structure $*_G$ on $\mcM(X)$ is given 
\[
\mu_1*_G\mu_2
:=j^{-1}\bigl(j(\mu_2)\circ j(\mu_1)\bigr).
\]

\end{definition}
It is clear that $q_*$ identifies $*_G$ with convolution of bi-$K$-invariant
measures on $G$.

\section{Basic definitions}
Let $F$ be a non-Archimedean local field with $\operatorname{char}F\ne2$,
and let $E/F$ be a quadratic \'etale algebra, with nontrivial
involution $e\mapsto\bar e$. Thus $E$ is either a quadratic field
extension or $F\oplus F$. Fix $\tau\in E^\times$ with
$\bar\tau=-\tau$, so that $E=F[\tau]$ and $a=\tau^2\in F^\times$.
I identify monic degree-$n$ polynomials with $F^n$ by their coefficient
map $\kk$.

For a central division $F$-algebra $L$ of degree $n$, put
$L_E=L\otimes_F E=L\oplus\tau^{-1}L$. Write $\Nrd_{L/F}$ and
$\Nrd^E=\Nrd_{L_E/E}$ for the reduced norms. Set
\[
 G_L=L_E^\times/F^\times,\quad K_L=L^\times/F^\times,\quad
 Y_E^L=G_L/K_L\simeq L_E^\times/L^\times,\quad
 X_E^L=K_L\backslash Y_E^L.
\]
Write $p_L:L_E^\times\to Y_E^L$ for the right quotient map.
The group $G_L$ is unimodular. Compactness of $K_L$ is proved below.
For $z\in L$, define
\[
 \chi_L(z)(u)=\Nrd_{L/F}(u-z),\qquad
 \Delta_L(z)=\Nrd_{L/F}(1-a^{-1}z^2),\qquad
 U_{L,E}=\{z\in L:\Delta_L(z)\ne0\}.
\]

\begin{claim}\label{proper}
The map $\kk\circ\chi_L:L\to F^n$ is proper.
\end{claim}

Let $L^+=L\cup\{\infty\}$ be its one-point compactification.
Properness gives a continuous extension
\[
 \chi_L^+:L^+\longrightarrow(F^n)^+,
 \qquad \chi_L^+(\infty)=\infty
\]

\begin{lemma}
For every $z\in L$,
\[
 \Nm_{E/F}\bigl(\Nrd^E(1+\tau^{-1}z)\bigr)=\Delta_L(z).
\]
\end{lemma}
\begin{proof}
Put $x=1+\tau^{-1}z$. Then $\bar x=1-\tau^{-1}z$ and
$x\bar x=1-a^{-1}z^2$. Compatibility of the reduced norm with
the involution, multiplication and scalar extension gives
\[
 \Nm_{E/F}(\Nrd^E(x))
 =\Nrd^E(x)\Nrd^E(\bar x)
 =\Nrd^E(x\bar x)=\Delta_L(z).
\]
For split $E$ the norm is the product of its two components.
In both cases $\Nm_{E/F}(b)\ne0$ if and only if $b\in E^\times$.
\end{proof}

\begin{corollary}\label{z}
The element $1+\tau^{-1}z$ belongs to $L_E^\times$ if and only if
$\Delta_L(z)\ne0$.
\end{corollary}

\begin{definition} \begin{enumerate}
\item 
 $\Xi=\{Q(u)^{n/d}:d\mid n,\ Q\in F[u]
       \text{ monic irreducible of degree }d\}$
\item $
 \Xi_E=\{P\in\Xi:P(\tau)P(-\tau)\ne0\}$
\item $
 \XX=\Xi\cup\{\infty\}$ and $
 \XX_E=\Xi_E\cup\{\infty\}$
\item 
The topology on $\XX$ is inherited from the one-point compactification
$(F^n)^+=F^n\cup\{\infty\}$ and $\XX_E$ has the induced subspace topology. \end{enumerate}\end{definition}

\section{The basic isomorphism}
Define a map  $\overline p_E^L:L_E^\times\to(F^n)^+$ by
\[
 \overline p_E^L(l_++\tau^{-1}l_-)=
 \begin{cases}
 \chi_L(l_-l_+^{-1}),&l_+\ne0,\\
 \infty,&l_+=0.
 \end{cases}
\]
This  map is invariant under left and right multiplication by
$L^\times$ and  therefore induces maps
\[
 r_E^L:Y_E^L\longrightarrow(F^n)^+,
 \qquad p_E^L:X_E^L\longrightarrow(F^n)^+.
\]

\begin{proposition}\label{local-quotient}
The map $p_E^L$ is a homeomorphism $X_E^L\xrightarrow{\sim}\XX_E$.
\end{proposition}
\begin{proof}
\emph{Image and fibers of the coefficient map.}
For $z\in L$, the algebra $A=F[z]$ is a field. If $Q$ is its
irreducible minimal polynomial and $d=[A:F]$, left multiplication
by $z$ on $L$ has characteristic polynomial $Q^{n^2/d}$.
After extending scalars to split $L$, the regular representation
of $M_n$ is a direct sum of $n$ standard representations, so this
same polynomial is $\chi_L(z)^n$. Unique factorization gives
$d\mid n$ and $\chi_L(z)=Q^{n/d}$.

Conversely, for any monic irreducible $Q$ of degree $d\mid n$, let
$A=F[u]/(Q)$ and  $M/A$  be the unramified extension of degree $n/d$.
Then $[M:F]=n$ and there exists an embedding $M\hookrightarrow L$.
The image $z$ of the
class of $u$ under $A\hookrightarrow M\hookrightarrow L$ has
$\chi_L(z)=Q^{n/d}$. Thus $\chi_L(L)=\Xi$.
Two elements have the same reduced characteristic polynomial exactly
when they have the same irreducible minimal polynomial. Skolem--Noether
\cite[Chapter IV, Theorem 2.10]{Milne}, applicable to the simple
source algebra $F[u]/(Q)$ without a separability assumption, shows
that they are conjugate under $L^\times$.

For $P=\chi_L(z)$, the norm identity gives
\begin{equation}\label{eq:allowed}
 \begin{aligned}
 P(\tau)P(-\tau)
 &=\Nrd^E((\tau-z)(-\tau-z))\\
 &=\Nrd_{L/F}(z^2-a)=(-a)^n\Delta_L(z).
 \end{aligned}
\end{equation}
Consequently the finite image of $\overline p_E^L$ is $\Xi_E$.
In the converse direction the required hypothesis on $Q$ is
$Q(\tau)\in E^\times$, equivalently $Q(\tau)Q(-\tau)\ne0$;
nonvanishing alone is insufficient when $E$ is split.
The inverse image of infinity is $\tau^{-1}L^\times$.
These observations also give bijectivity on the double quotient.

\emph{The compact quotient.}
Every class of $K_L=L^\times/F^\times$ has a representative in the
compact sphere $S$ from the proof of Claim~\ref{proper}, so $K_L$ is
compact. Conjugation extends continuously to $L^+$, fixing infinity.
Indeed it preserves the reduced norm, and Claim \ref{proper}
shows uniformly that conjugates of elements escaping compact sets
still escape compact sets. The continuous map $\chi_L^+$ has image
$\XX$ and fibers precisely the $K_L$-orbits, including infinity.
Hence it induces a continuous bijection
\begin{equation}\label{eq:compact-quotient}
 K_L\backslash L^+\xrightarrow{\sim}\XX.
\end{equation}
This is a homeomorphism because its source is compact and its
target is Hausdorff.\end{proof}
\begin{lemma} The map $r_E^L:Y_E^L\to\XX_E$ is proper.
\end{lemma}
\begin{proof}

For $g=l_++\tau^{-1}l_-\in L_E^\times$ with $l_+\ne0$, put
$z=l_-l_+^{-1}$. Then $g=(1+\tau^{-1}z)l_+$, and
Corollary~\ref{z} gives a chart
\[
 U_{L,E}\xrightarrow{\sim}\{p_L(g):l_+\ne0\},
 \qquad z\longmapsto p_L(1+\tau^{-1}z).
\]
Its inverse is induced by the continuous, right-$L^\times$-invariant
function $g\mapsto l_-l_+^{-1}$.
If $l_+=0$, then $p_L(g)=p_L(\tau^{-1})$, denoted by infinity.
On the open set $l_-\ne0$ there is a second chart
\[
 w=l_+l_-^{-1},\qquad w\longmapsto p_L(w+\tau^{-1}),\qquad
 \Nrd(1-aw^2)\ne0.
\]
It contains $w=0$, representing infinity, and its transition is
$w=z^{-1}$. Both are charts for the quotient topology, since the
quotient map is open and the coordinate domains upstairs are open
and right invariant. Since 
$\Nrd(z^{-1})=\Nrd(z)^{-1}$ Claim \ref{proper} implies that 
$
 \|z\|\longrightarrow\infty\quad\Longleftrightarrow\quad
 \|z^{-1}\|\longrightarrow0.
$
Thus the two charts identify $Y_E^L$ with
$U_{L,E}\cup\{\infty\}\subset L^+$, including its topology at infinity.
The left $K_L$-action becomes conjugation and fixes infinity.
Equation~\eqref{eq:allowed} gives
\begin{equation}\label{eq:preimage}
 (\chi_L^+)^{-1}(\XX_E)=U_{L,E}\cup\{\infty\}=Y_E^L,
 \qquad \chi_L^+|_{Y_E^L}=r_E^L.
\end{equation}
This subset is saturated. Restricting \eqref{eq:compact-quotient}
gives the claimed homeomorphism $p_E^L$. For the topology, an orbit
map is open, and for a saturated subset $Z'$ one has
$\pi(V\cap Z')=\pi(V)\cap\pi(Z')$ for every open $V$.

Finally, a compact subset $B\subset\XX_E$ is compact and closed in
$(F^n)^+$. Its inverse image under $\chi_L^+$ is therefore closed in
compact $L^+$ and, by \eqref{eq:preimage}, equals $(r_E^L)^{-1}(B)$
inside $Y_E^L$. This proves properness.
\end{proof}

The topology at infinity is inherited from $(F^n)^+$. A family of
coefficient vectors converging to a deleted finite parameter does not
converge to infinity in $\XX_E$.

\section{The first conjecture}
\label{sec:local-conjectures}

Let $\MM_E$ be the space of compactly supported complex Radon measures
on the common quotient space $\XX_E$ and $*_E^L:\MM_E\times\MM_E\to\MM_E$ be the product as in Definition \ref{product}.

\begin{proposition}\label{prop:local-commutativity}
The product $*_E^L$ is commutative, and $*_E^L=*_E^{L^{\mathrm{op}}}$
under the common identification with $\MM_E$.
\end{proposition}
\begin{proof}
The anti-automorphism $h\mapsto\bar h^{-1}$ of $G_L$ preserves $K_L$
and fixes every double coset. For a representative $1+\tau^{-1}z$
in the finite chart, this follows from
\[
 \overline{(1+\tau^{-1}z)}^{-1}
   =(1+\tau^{-1}z)(1-a^{-1}z^2)^{-1},
\]
since the right factor belongs to $L^\times$. At infinity,
$\overline{\tau^{-1}}^{-1}=-\tau$ belongs to the same double coset
as $\tau^{-1}$, since $-\tau=(-a)\tau^{-1}$.
By uniqueness of the lift in Lemma~\ref{iso}, this anti-automorphism
acts trivially on bi-$K_L$-invariant measures. Its pushforward reverses
convolution, which proves commutativity.

\end{proof}
\begin{conjecture}\label{local}
The convolution product $*_E^L$ on $\MM_E$ does not depend on the
choice of the central division $F$-algebra $L$ of degree $n$.
\end{conjecture}

\begin{lemma}\label{split}
Conjecture~\ref{local} holds for $E=F\oplus F$ in every degree.
\end{lemma}
\begin{proof}
In this case $G_L=(L^\times\times L^\times)/F^\times$ and
$K_L=L^\times/F^\times$, with both diagonal embeddings. The map
$(g_1,g_2)\mapsto g_1g_2^{-1}$ identifies
$Y_E^L$ with $L^\times$, and the left $K_L$-action
becomes conjugation. Thus $\MM_E$ identifies with the conjugation-invariant compactly
supported measures on $L^\times$ and $\ast ^L_E$ is the convolution product on $\MM_E$. In \cite{DKV} for fields $F$ of characteristic $0$ and in  \cite[Theorem~5.1]{Bad} for fields of positive characteristic the independence of $L$ is proven  only for the restrictions of products  $\ast ^L_E$ on the subspace of smooth measures. But since this subspace is dense these results imply the validity of Lemma \ref{split}.
\end{proof}

\section{Bundles, degrees, and half-densities}
\label{sec:bundles}

\subsection{Geometric data and the degree}
Fix a smooth projective geometrically integral curve $C/F$ of genus
$g_C$ and an effective divisor $M\subset C$ (possibly empty) such that there exists of a
line bundle $\KK$ of degree one on $C$.
Write $K_C=\Omega_C^1$ for the canonical line bundle.
When $M\ne\varnothing$,
also fix
\begin{equation}
 \nu:\KK|_M\xrightarrow{\sim}\OO_M.
 \label{level-trivialization}
\end{equation} \begin{remark} The existence of $\KK$
 does not imply the existence of an
$F$-rational point of $C$.
\end{remark}

Let  $\GG^L$ be the algebraic group, whose points over a
commutative $F$-algebra $R$ are $(L\otimes_F R)^\times$.
A $\GG^L$-bundle is a locally
free rank-one right module over $L\otimes_F\OO_C$. Equivalently,
it is a vector bundle $\EE$ of rank $n^2$ endowed with a unital
algebra action
$L^{\mathrm{op}}\to\End_{\OO_C}(\EE)$.

Write $\BB_M^L=\Bun_{\GG^L,M}(C)$ for the stack of such bundles
equipped with a  trivialization
\[
 \iota:\EE|_M\xrightarrow{\sim}L\otimes_F\OO_M.
\] of right  $L$-modules.
The reduced-norm character defines a line bundle
$\det_L\EE$ on $C$; put $\delta(\EE)=\deg(\det_L\EE)$.

\smallskip
\noindent\textbf{The Morita bundle.}
Choose a separable closure $\bar F$ of $F$ and a splitting
\[
 L\otimes_F\bar F\xrightarrow{\sim}M_n(\bar F).
\]
Put
\[
 X=C_{\bar F},\qquad
 \mathcal A_X=L\otimes_F\OO_X\simeq M_n(\OO_X),\qquad
 S=\OO_X^{\oplus n}.
\]
The bundle $S$ has the standard left $\mathcal A_X$-action on
column vectors, whereas $\EE_{\bar F}$ is a right
$\mathcal A_X$-module. Define
\[
 \boxed{\displaystyle
 \widehat\EE=\EE_{\bar F}\otimes_{\mathcal A_X}S.}
\]
The tensor product is over the sheaf of algebras $\mathcal A_X$.
Since $\EE_{\bar F}$ is locally isomorphic to $\mathcal A_X$
as a right module, locally
\[
 \widehat\EE\simeq\mathcal A_X\otimes_{\mathcal A_X}S\simeq S.
\]
Thus $\widehat\EE$ is an ordinary vector bundle of rank $n$ on
$X$. I call it the \emph{Morita bundle} associated to $\EE$
and the chosen splitting.

\smallskip
\noindent\textbf{The matrix-column description.}
Let $e_{11}\in M_n(\bar F)$ be the matrix with a single $1$ in
position $(1,1)$ and zeros elsewhere. The left module
$\mathcal A_Xe_{11}$ consists of matrices whose only possibly
nonzero column is the first, and hence is isomorphic to $S$.
Consequently,
\[
 \widehat\EE\simeq\EE_{\bar F}e_{11}.
\]
This is the usual matrix-algebra form of Morita equivalence;
see \cite[Lemma 11.4.5, Tag 074D]{Stacks}.
More explicitly, choose local right-$\mathcal A_X$-module
trivializations of $\EE_{\bar F}$. Their transition maps are
left multiplications
\[
 B\longmapsto g_{ij}B,\qquad
 g_{ij}\in\GL(n,\OO_X(U_i\cap U_j)).
\]
Each column transforms by $v\mapsto g_{ij}v$. The bundle
$\widehat\EE$ therefore has transition matrices $g_{ij}$.

\smallskip
\noindent\textbf{The inverse equivalence.}
For an ordinary vector bundle $V$ on $X$, the sheaf
\[
 V\otimes_{\OO_X}S^\vee
 \simeq\mathcal{H}\!om_{\OO_X}(S,V)
\]
has a right $\mathcal A_X=\mathcal{E}\!nd_{\OO_X}(S)$-action
by precomposition: $(f\cdot a)(s)=f(as)$. The functors
\[
 \mathcal M\longmapsto\mathcal M\otimes_{\mathcal A_X}S,
 \qquad
 V\longmapsto V\otimes_{\OO_X}S^\vee
\]
are mutually inverse Morita equivalences. They restrict to an
equivalence between locally free rank-one right
$\mathcal A_X$-modules and ordinary rank-$n$ vector bundles.
For the bundle under consideration, they give
\[
 \EE_{\bar F}\simeq
 \widehat\EE\otimes_{\OO_X}S^\vee
 \simeq\widehat\EE^{\oplus n}.
\]
The first isomorphism respects the right $\mathcal A_X$-action;
the second is an isomorphism of ordinary vector bundles using
the standard basis of $S$. The equivalence also identifies
\[
 \mathcal{E}\!nd_{\mathcal A_X}(\EE_{\bar F})
 \simeq\mathcal{E}\!nd_{\OO_X}(\widehat\EE).
\]
A right-module framing of $\EE$ along $M$ becomes the usual
framing $\widehat\EE|_{M_{\bar F}}\simeq\OO_{M_{\bar F}}^{\oplus n}$.

\smallskip
\noindent\textbf{Determinant and degree.}
After splitting $L$, the reduced norm is the ordinary determinant.
The transition functions of $\det\widehat\EE$ are therefore
$\det(g_{ij})=\Nrd(g_{ij})$, which are the transition functions
of the reduced-norm line. Thus
\[
 \det\widehat\EE\simeq(\det_L\EE)_{\bar F},\qquad
 \deg\widehat\EE=\delta(\EE).
\]
The direct-sum description of $\EE_{\bar F}$ gives
$\deg\EE=n\deg\widehat\EE$. Moreover, tensoring by a line
bundle commutes with the Morita construction:
\[
 \widehat{\EE\otimes\KK}
 \simeq\widehat\EE\otimes\KK_{\bar F}.
\]
Since $\widehat\EE$ has rank $n$ and $\KK$ has degree one,
I obtain
\begin{equation}
 \deg\EE=n\,\delta(\EE),\qquad
 \delta(\EE\otimes\KK)=\delta(\EE)+n.
 \label{degree-identities}
\end{equation}

\smallskip
\noindent\textbf{Dependence on the splitting.}
Two splittings of $L\otimes_F\bar F$ differ by an inner
automorphism of $M_n(\bar F)$. A conjugating matrix identifies
the resulting Morita bundles, but it is determined only up to
a nonzero scalar. Thus only the isomorphism class of $\widehat\EE$
is independent of the splitting. Thus the  construction does not in general supply a descent data for
$\widehat\EE$ to $C$. On the other hand the  line bundle $\det_L\EE$
is intrinsically defined on $C$, so its degree is defined over $F$.

For $r\in\Z$, define $\BB_{M,r}^L\subset\BB_M^L$ to be the
full substack of bundles with $\delta(\EE)=r$.
These degree substacks are open and closed.

\begin{lemma}\label{degree-divisibility}
For every $F$-object of $\BB_M^L$, one has $\delta(\EE)\in n\Z$.
\end{lemma}
\begin{proof}
The finite-dimensional spaces $H^i(C,\EE)$ are right $L$-modules.
Their $F$-dimensions are divisible by $n^2$. Riemann--Roch gives
$\deg\EE=\chi(\EE)-n^2(1-g_C)\in n^2\Z$.
Now use \eqref{degree-identities}.
\end{proof}

\subsection{Semistability}
For a nonzero vector bundle $V$, its slope is
$\mu(V)=\deg(V)/\rk(V)$. It is semistable if
$\mu(W)\leq\mu(V)$ for every nonzero proper subbundle
$W\subset V$. Semistability of an $L$-bundle means this condition
for its underlying vector bundle; geometric semistability means
semistability after base change to an algebraic closure of $F$.
For a closed point $c\in C$, write $k(c)$ for its residue field
and $\deg(c)=[k(c):F]$.

\begin{lemma}\label{ss}
Let $V$ be a semistable vector bundle on $C$, and let $c$ be a
closed point. An endomorphism $a$ of $V$ that vanishes on $V_c$
is zero.
\end{lemma}
\begin{proof}
The endomorphism factors through $V(-c)$. Both bundles are
semistable and $\mu(V(-c))=\mu(V)-\deg(c)<\mu(V)$, so
$\Hom(V,V(-c))=0$.
\end{proof}

\begin{proposition}\label{se}
Assume that the gcd $\ind(C)$ of the degrees of closed points of
$C$ is prime to $n$. Then every $\GG^L$-bundle defined over $F$
is geometrically semistable.

\end{proposition}
\begin{proof}
The Harder--Narasimhan filtration commutes with extensions
of the ground field, including purely inseparable
\cite[Theorem 1.3.7 and Corollary 1.3.8]{HL}. Thus a nontrivial
first Harder--Narasimhan subbundle is defined over $F$.
It is invariant under
all automorphisms given by right multiplication by $L^\times$,
and hence under $L$. Let its ordinary rank be $r$, where
$0<r<n^2$. At every closed point $c$, its fiber regarded as an
$F$-vector space is a right $L$-module. Thus
$n^2\mid r[k(c):F]$ for every $c$, and consequently
$n^2\mid r\ind(C)$. Coprimality forces $n^2\mid r$, a contradiction.
\end{proof}
\begin{remark}The existence of   $\KK$ implies $\ind(C)=1$.
\end{remark}

Let $\BB_{M,\sst}^L$ denote the geometrically semistable locus, and set
$ \BB_{M,\sst,r}^L =\BB_{M,\sst}^L\cap\BB_{M,r}^L$.
The inclusion induces a bijection
on isomorphism classes of $F$-objects under the stated hypotheses;
it is not an equality of stacks.
For $M\ne\varnothing$, a framing-preserving automorphism is the
identity by Lemma~\ref{ss}. Under this nonempty-level assumption,
$\BB_{M,\sst,r}^L$ is a smooth algebraic space.
\begin{lemma}
If $M\ne\varnothing$, the algebraic space
$\BB_{M,\sst,r}^L$ is separated.
\end{lemma}
\begin{proof}
Let $R$ be a discrete valuation ring with uniformizer $\varpi$,
and let $V_1,V_2$ be two framed geometrically semistable families
over $C_R$, of the same reduced degree. Let $\varphi$ be a
framing-preserving isomorphism of their generic fibers.
Choose the least nonnegative integer $k$ for which
$\varpi^k\varphi$ extends to a morphism $a:V_1\to V_2$.
Such an integer exists by properness of $C_R$ over $R$.
If $k>0$, minimality implies that the special-fiber morphism
$\bar a$ is nonzero. Compatibility with the framings gives
$\bar a|_M=0$, so $\bar a$ factors through $V_{2,s}(-M)$.
But $V_{1,s}$ and $V_{2,s}$ are semistable of the same slope,
while twisting by the nonempty effective divisor $M$ decreases
the slope. Hence
$\Hom(V_{1,s},V_{2,s}(-M))=0$, a contradiction.
Thus $\varphi$ extends. Applying the same argument to
$\varphi^{-1}$ gives an inverse extension. Uniqueness follows
because two morphisms of these flat families that agree on the
generic fiber agree everywhere. The valuative criterion proves
separatedness.
\end{proof}

\subsection{The canonical line and its positive square root}
Put
\[
 \operatorname{ad}(\EE)=
       \End_{L\otimes_F\OO_C}(\EE).
\]
The tangent complex of the framed bundle stack at $\EE$ is
$R\Gamma(C,\operatorname{ad}(\EE)(-M))[1]$. Our determinant convention is
\[
 \det R\Gamma(C,V)
 =\det H^0(C,V)\otimes\det H^1(C,V)^{-1}.
\]
The canonical line is the determinant of the cotangent complex,
and therefore has fiber
\begin{equation}
 \omega_{\BB_M^L,\EE}
   =\det R\Gamma(C,\operatorname{ad}(\EE)(-M)).
 \label{canonical-line}
\end{equation} \begin{remark}This construction 
$F$ and does not depend on a choice of  a
splitting of $L$.\end{remark}

For an algebraic line bundle $\mathcal L$, the notation $|\mathcal L|$
means the associated complex line with positive transition functions
obtained from $|\cdot|_F$. All real tensor powers of a positive
density line are defined by taking the corresponding real powers
of its positive transition functions; negative powers use the dual
line. I use
\begin{equation}
 \dens_{\BB_M^L}^{1/2}=|\omega_{\BB_M^L}|^{1/2}.
 \label{half-density}
\end{equation}

Thus no algebraic square root of $\omega_{\BB_M^L}$ is required.
For a smooth algebraic space $Y$ and a smooth complex coefficient
line $\mathcal L$ on $Y(F)$, write
$\SSS(Y,\mathcal L)=C_c^\infty(Y(F),\mathcal L)$.
Here $\omega_Y$ is the top exterior power of the cotangent bundle,
and $\SSS_{1/2}(Y)=\SSS(Y,|\omega_Y|^{1/2})$.
It has the Hermitian form
$\langle f_1,f_2\rangle=\int_{Y(F)}f_1\overline{f_2}$.

The absolute value of a nowhere-vanishing algebraic section gives
a smooth local frame of the corresponding absolute-value line.
At zeros of an algebraic section its absolute value need not be
locally constant, so smooth sections are defined independently of
such absolute values.

\subsection{Schwartz spaces of stacks and periodicity}
Following \cite[Definition~2.3]{BK}, a stack is called admissible
if it is covered by open substacks admitting quotient presentations
$[X/G]$, with $X$ a smooth $F$-variety and $G$ an affine algebraic
$F$-group. A presentation of the form $[X/\GL(r)]$ is called
an admissible presentation. The bundle stacks used here admit such
presentations on their finite-type open substacks.
For a smooth representable morphism $q:H\to Y$, write
$\omega_q=\omega_{H/Y}=\det\Omega^1_{H/Y}$ for its relative
canonical line; on stacks this definition is read in smooth charts.

For an smooth admissible algebraic stack $\mathfrak Y$, use the Schwartz half-density space $\SSS_{1/2}(\mathfrak Y)$ is defined in 
 \cite[\S2.2]{BK}. For a stack
locally of finite type it is the colimit over open substacks of
finite type.
For a group $G$ acting on a complex vector space $V$, algebraic
coinvariants mean
$V_G=V/\operatorname{span}_{\C}\{gv-v:g\in G,\ v\in V\}$;
no topological closure is taken. Thus the subscript above has this
meaning, and $\mathcal L=|\omega_{\mathfrak Y}|^{1/2}$ gives
$\SSS_{1/2}(\mathfrak Y)$. The relative-density factor is the
inverse absolute determinant line of the Lie algebra of $\GL(r)$;
this is the atlas convention in \cite[\S2.3]{BKP}.

Tensoring by $(\KK,\nu)$ defines
\[
 \psi(\EE,\iota)=(\EE\otimes\KK,\iota\otimes\nu).
\]
For $M=\varnothing$, the same formula omits $\iota$ and $\nu$.
Its action on Schwartz half-densities is transport of sections,
equivalently pullback by $\psi^{-1}$, and is also denoted by $\psi$.

\begin{definition}\label{global-space}
The periodic space is the coinvariant space
\begin{equation}
 \SSS_M^L
 =
 \SSS_{1/2}(\BB_M^L)_{\langle\psi\rangle}
 =
 \SSS_{1/2}(\BB_M^L)/
       \langle\psi f-f:f\in\SSS_{1/2}(\BB_M^L)\rangle.
 \label{periodic-space}
\end{equation}
\end{definition}

Because $\psi$ shifts $\delta$ by $n$, there is an identification
\begin{equation}
 \SSS_M^L\simeq
 \bigoplus_{r=0}^{n-1}\SSS_{1/2}(\BB_{M,r}^L).
 \label{fundamental-degrees}
\end{equation}
For $M\ne\varnothing$, the semistability result identifies these
spaces with Schwartz half-densities on the framed semistable
algebraic spaces. Indeed, on any admissible presentation all
$F$-points lie over the semistable locus; restricting to that
locus leaves the points, topology, density lines, and the
coinvariants defining the Schwartz space unchanged. Equivalently,
\eqref{periodic-space} consists of
$\psi$-periodic families that are compactly supported on each
degree component. The identification is given by periodization;
the sum of translates is locally finite in the degree index.
For $M\ne\varnothing$, the Hermitian form is the sum of the integrals over the
fundamental degrees in \eqref{fundamental-degrees}, and is
independent of their chosen representatives.

Lemma~\ref{degree-divisibility} shows that, on $F$-objects, only
the degree-zero residue class contributes.

\section{Fundamental Hecke correspondences}
\label{sec:hecke}

\subsection{The modification type}
Let $\DD_M$ be the set of effective divisors $D\subset C$ finite
\'etale over $F$ and disjoint from $M$; each such divisor is
reduced. Write $I_D\subset\OO_C$ for the ideal sheaf of $D$,
$|D|$ for its set of closed points, and
$\deg D=\dim_F\Gamma(D,\OO_D)=\sum_{c\in|D|}[k(c):F]$.
For the empty divisor this degree is zero.
For $D\in\DD_M$ define the Hecke correspondence 
 $\HH_D^L$ to be the stack of
right-$L$-linear inclusions
\[
 \EE_-\hookrightarrow\EE_+
\]
with compatible framings along $M$, such that
\begin{equation}
 I_D(\EE_+/\EE_-)=0,\qquad
 \EE_+/\EE_-\text{ is locally free of rank }n
                           \text{ over }\OO_D.
 \label{hecke-quotient}
\end{equation}
Write
\[
 q_-(\EE_-\subset\EE_+)=\EE_-,
 \qquad
 q_+(\EE_-\subset\EE_+)=\EE_+.
\]
Our operator convention is to integrate $q_+^*f$ along $q_-$.

\begin{proposition}\label{hecke-proper}
On the full bundle stacks, both projections $q_-$ and $q_+$
are representable, smooth, and proper, of relative dimension
$(n-1)\deg D$. Moreover,
\[
 \delta(\EE_+)-\delta(\EE_-)=\deg D.
\]
\end{proposition}
\begin{proof}
Pass to a faithfully flat extension splitting $L$ and $D$;
for the notation below I may use $\bar F$ and write
$X=C_{\bar F}$, $\mathcal A_X=L\otimes_F\OO_X$, and
$S=\OO_X^{\oplus n}$ as in Section~\ref{sec:bundles}.
Let $i:D_{\bar F}\hookrightarrow X$ be the closed immersion.
Since $I_D(\EE_+/\EE_-)=0$, write
\[
 (\EE_+/\EE_-)_{\bar F}=i_*Q_D,
\]
where $Q_D$ is a right $\mathcal A_X|_{D_{\bar F}}$-module
of ordinary rank $n$ over $\OO_{D_{\bar F}}$. Morita equivalence
is exact and gives
\[
 0\longrightarrow\widehat\EE_-
 \longrightarrow\widehat\EE_+
 \longrightarrow i_*\lambda\longrightarrow0,
\]
where
\[
 \lambda=Q_D\otimes_{\mathcal A_X|_{D_{\bar F}}}
                       S|_{D_{\bar F}}.
\]
The inverse equivalence identifies
$Q_D\simeq\lambda\otimes S^\vee|_{D_{\bar F}}$; hence
$\lambda$ has rank one over $\OO_{D_{\bar F}}$. Thus this is
an elementary modification of ordinary rank-$n$ vector bundles,
with a one-dimensional quotient at each point of $D_{\bar F}$.
The fiber over $\EE_+$ is the product of the projective spaces
of one-dimensional quotients of the rank-$n$ fibers. The fiber
over $\EE_-$ is the product of the projective spaces of lines
determining inverse elementary modifications. These descriptions
give representability, smoothness, properness, and relative
dimension $(n-1)\deg D$, all of which descend. Finally,
$i_*\lambda$ has total length $\deg D$, so
\[
 \delta(\EE_+)-\delta(\EE_-)
 =\deg\widehat\EE_+-\deg\widehat\EE_-
 =\deg D.
\]
\end{proof}
\begin{remark}
Properness here concerns the full stacks. Restricting both endpoints
to semistable loci need not preserve algebraic properness.
But, as follows from Proposition~\ref{se}, in the case when $M\neq \emp$ 
the induced map on the topological spaces of $F$-points is proper. \end{remark}
\begin{lemma}
Let $c\in C\setminus M$ be a closed separable point of degree $d$ such that
$L\otimes_F k(c)$ is nonsplit. Then $\HH_c^L$ has no $F$-points.
More generally, the same holds for $\HH_D^L$ whenever $c\in|D|$.
\end{lemma}

\begin{proof} A
right module of $k(c)$-dimension $n$ over that central simple
algebra exists only in the split case. Restriction to $k(c)$
multiplies the local Brauer invariant by $d$. Since the invariant
of $L$ has order $n$, the algebra splits exactly when $n\mid d$.
Thus $n\nmid d$ rules out the required quotient at $c$.\end{proof}

\subsection{The density comparison}
For a closed separable point $c$, let $\mathscr T_c$ be its positive
cotangent-density line, using the normalized absolute value of
$k(c)$. Equivalently, it is the absolute norm line of
$K_C|_c$ over $F$. For a divisor $D$ put
$\mathscr T_D=\bigotimes_{c\in|D|}\mathscr T_c$.
Restriction of scalars uses the trace pairing for the finite
\'etale residue algebras.
Explicitly, for $b\in k(c)^\times$ and a one-dimensional
$k(c)$-space $V$, our conventions are
\[
 |b|_{k(c)}=|\Nm_{k(c)/F}(b)|_F,\qquad
 \Nm_{k(c)/F}(V)=\det_F(V)\otimes\det_F(k(c))^{-1}.
\]
Thus $\mathscr T_c=|\Nm_{k(c)/F}(K_C|_c)|$; no degree root is
taken in the absolute value.

\Needspace{9\baselineskip}
\begin{proposition}\label{critical-density}
There is a canonical positive density isomorphism
\begin{equation}
 q_+^*\dens_{\BB_M^L}^{1/2}
 \xrightarrow{\ \alpha_D\ }
 |\omega_{q_-}|\otimes q_-^*\dens_{\BB_M^L}^{1/2}
       \otimes\mathscr T_D^{-(n-1)/2}.
 \label{density-comparison}
\end{equation}
It is compatible with modifications at disjoint points, with
change of framing, and with tensoring both bundles by a line
bundle on $C$.
\end{proposition}
\begin{proof}
In the split setting this is the positive-density form of
\cite[Theorem 1.1, formulas (1.3)--(1.4)]{EFK}.
The following determinant calculation, starting from
\eqref{canonical-line}, establishes the version needed here,
including its descent to inner forms and closed separable points.
Work first at a split rational point and write $\ell=K_C|_c$.
If $V$ is the larger Morita bundle, let $Q$ be its one-dimensional
quotient and $U=\ker(V_c\to Q)$. For the smaller bundle $W$ there
is an exact sequence
\[
 0\longrightarrow Q\otimes\ell
 \longrightarrow W_c\longrightarrow U\longrightarrow0.
\]
The relative tangent spaces for $q_+$ and $q_-$ are, respectively,
$\Hom(U,Q)$ and $\Hom(Q\otimes\ell,U)$.
Their determinant duals give the two relative canonical lines
\[
 \omega_{q_+}=\det U\otimes Q^{-(n-1)},\qquad
 \omega_{q_-}=(\det U)^{-1}\otimes Q^{n-1}
                                     \otimes\ell^{n-1}.
\]
Their product is $\ell^{n-1}$. Comparing
$\omega_{\HH}=q_\pm^*\omega_{\BB}\otimes\omega_{q_\pm}$
therefore gives
\[
 q_+^*\omega_{\BB}
 \simeq q_-^*\omega_{\BB}\otimes\omega_{q_-}^{\otimes2}
                                      \otimes\ell^{-(n-1)}.
\]
Take absolute values and the unique positive square root.
The construction is intrinsic in the modification and descends
to inner forms. At a closed separable point $c$, put $k=k(c)$ and
let $T$ be either of the relative tangent spaces over $k$.
Writing $T^\vee_k$ and $T^\vee_F$ for the duals over the indicated
fields, the trace pairing gives $T^\vee_F\simeq T^\vee_k$ as
$F$-spaces. As $\dim_k T=n-1$, the determinant formula for
restriction of scalars gives
\[
 \det_F(T^\vee_F)
 \simeq \Nm_{k/F}\bigl(\det_k(T^\vee_k)\bigr)
          \otimes(\det_F k)^{\otimes(n-1)}.
\]
The product of the two relative canonical lines thus has a factor
$(\det_F k)^{\otimes 2(n-1)}$. The determinant of the nondegenerate
trace pairing on $k$ canonically trivializes $(\det_F k)^{\otimes 2}$.
The remaining factor is
$\Nm_{k/F}(K_C|_c)^{\otimes(n-1)}$. Taking absolute values and
multiplying over disjoint points gives \eqref{density-comparison}.
\end{proof}

In particular, for even $n$ the half-integral exponent in
\eqref{density-comparison} is a real tensor power of a positive density
line; it involves no choice of an algebraic theta characteristic.

For every closed separable point $c$ outside $M$, fix a positive
element $\vartheta_c\in\mathscr T_c$. Make these choices once for
the curve, independently of $L$, and use the tensor products
$\vartheta_D$ for divisors. Contracting the last factor of
\eqref{density-comparison} against
$\vartheta_D^{(n-1)/2}$ gives a scalar-valued density comparison
$\alpha_{D,\vartheta}$.

\subsection{Integration and the Hecke operators}
For clarity, the elementary integration statement for spaces is
the following. If $q:H\to Y$ is smooth and proper, $q':H\to Y'$
is smooth, and
\[
 \beta:q'^*\mathcal L'\xrightarrow{\sim}
                |\omega_q|\otimes q^*\mathcal L
\]
is an isomorphism of smooth density lines, then
$q_!(\beta q'^*f)$ is a smooth section on $Y$, where $q_!$
denotes integration along the fibers against the relative density.
If $q'$ is also proper, this operation preserves compact support.
Local product coordinates prove smoothness, and properness
supplies the support statements. Compactness of individual
fibers alone is not used as a substitute for these hypotheses.

For the full bundle stacks, use the Schwartz spaces and smooth
representable pushforward of \cite[\S\S2.2--2.3]{BK}; the Hecke
construction in \cite[\S5.5]{BK} provides the corresponding framework.
For the proper representable pullback, let
$q':\mathfrak H\to\mathfrak Y$ be a proper representable morphism
between smooth admissible stacks. Work first over a finite-type
open substack with presentation $p:X\to\mathfrak Y=[X/\GL(r)]$.
Put $H_X=\mathfrak H\times_{\mathfrak Y}X$, with projections
$p_H:H_X\to\mathfrak H$ and $q'_X:H_X\to X$.
Then $q'_X$ is proper, $\mathfrak H=[H_X/\GL(r)]$, and base change gives
\[
 |\omega_{H_X/\mathfrak H}|
 \simeq(q'_X)^*|\omega_{X/\mathfrak Y}|.
\]
Thus ordinary pullback carries the atlas coefficient line
$p^*\mathcal L\otimes|\omega_{X/\mathfrak Y}|$ to
$p_H^*(q'^*\mathcal L)\otimes|\omega_{H_X/\mathfrak H}|$.
It preserves compact support and is $\GL(r,F)$-equivariant,
so it defines $q'^*$ on coinvariants. If $H_X$ is an algebraic
space, use its \'etale charts. Common smooth refinements and
Cartesian base change for fiber integration give independence
of the presentation. Properness gives finite-type inverse images
of finite-type open substacks, so the construction passes to
the defining colimit.
Combining this pullback with smooth representable pushforward and
the above density normalization, define
\begin{equation}
 T_D^L f=(q_-)_!\bigl(\alpha_{D,\vartheta}\,q_+^*f\bigr)
 \quad\text{on }\SSS_{1/2}(\BB_M^L).
 \label{global-hecke}
\end{equation}
For nonempty level, its description on the semistable $F$-points
is the ordinary integral over the compact fiber of $q_-$.
The density calculation and diagram are compatible with $\psi$,
so \eqref{global-hecke} descends to $\SSS_M^L$.
Before passing to periodic coinvariants, our operator convention
sends the degree-$r$ summand to the degree-$(r-\deg D)$ summand.
I allow the empty divisor $D=0$; its correspondence is the identity,
its density normalization is $1$, and $T_0^L=\id$.

\begin{proposition}\label{global-commuting}
The operators $T_D^L$ commute. If $D_1,D_2$ have disjoint supports,
then
\[
 T_{D_1+D_2}^L=T_{D_1}^LT_{D_2}^L.
\]
\end{proposition}
\begin{proof}
Modifications at distinct points can be performed in either
order and give canonically the same correspondence. The density
isomorphisms factor in the same way, and proper pushforward and
Fubini give the identity. Every reduced divisor is a sum of
distinct closed points. There is one fundamental generator
at each point, so all these products commute.
\end{proof}

\section{The global conjecture}
\label{sec:conjecture}

\subsection{The projective line with two level points}
Take $C=\PP^1_F$ and a quadratic \'etale divisor $M$ with a fixed
identification $\Gamma(M,\OO_M)\simeq E$. Choose $\infty\in C(F)$
outside $M$ and an affine coordinate $v$ on its complement. Write
$v|_M=b+c\tau$ with $b,c\in F$. Since $v|_M$ generates $E$,
$c\ne0$, and $t=(v-b)/c$ satisfies $t|_M=\tau$ and $M=V(t^2-a)$.

Take $\KK=\OO(\infty)$ with its canonical section $s$, and choose
its level trivialization $\nu$ by
\begin{equation}
 \nu(s|_M)=\tau^{-1}.
 \label{p1-period}
\end{equation}
Every semistable $F$-bundle is $\OO(r)\otimes_F L$. Indeed its
underlying bundle on $\PP^1_F$ has equal splitting degrees,
and after twisting to degree zero, global sections identify it
with a rank-one right $L$-module tensored with $\OO$.
Its reduced degree is $nr$.

Framings of $\OO\otimes_F L$ along $M$, modulo its automorphisms,
give
\begin{equation}
 \BB_{M,\sst,0}^L(F)\simeq L_E^\times/L^\times=Y_E^L.
 \label{p1-space}
\end{equation}
This is an analytic identification. At the identity framing the
differential of the framing orbit map is the connecting map
$L_E\to H^1(C,L\otimes_F\OO_C(-M))$ for restriction to $M$.
It induces an isomorphism on $L_E/L$ since $H^1(C,\OO_C)=0$.
Equivariance and the analytic inverse function theorem give the assertion
at every framing.
Choose a $G_L$-invariant positive smooth density $dx_L$ on $Y_E^L$.
Such a density exists because $G_L$ is unimodular and $K_L$ is compact.
Put $\SSS_E^L=C_c^\infty(Y_E^L)$ and write
\[
 I_L:\SSS_E^L\xrightarrow{\sim}
 C_c^\infty(Y_E^L,|\omega_{Y_E^L}|^{1/2}),
 \qquad I_L(f)=f\,dx_L^{1/2}.
\]
Using $\psi$ to identify the other nonempty degree components with
degree zero gives
\begin{equation}
 \SSS_M^L\simeq
 C_c^\infty(Y_E^L,|\omega_{Y_E^L}|^{1/2})
 \xleftarrow[\ I_L\ ]{\sim}\SSS_E^L.
 \label{p1-schwartz}
\end{equation}

For $\eta\in\MM_E$, let $\widetilde\eta$ be its bi-$K_L$-invariant
lift to $G_L$, with the orbital probability normalization of
Lemma~\ref{iso}. Define
\[
 (A_\eta^L f)(gK_L)=
 \int_{G_L}f(gh^{-1}K_L)\,d\widetilde\eta(h),
 \qquad f\in\SSS_E^L.
\]
This operator preserves compact support and local constancy and
commutes with the left $G_L$-action. Under the embedding
$f\mapsto f\,dx_L$ into measures, it is the restriction of the
operator of Section~1:
\[
 T_{\widetilde\eta}(f\,dx_L)=(A_\eta^L f)\,dx_L.
\]
Thus the same convention acts on half-densities as
$I_L A_\eta^L I_L^{-1}$. It does not depend on multiplying $dx_L$
by a positive constant. The action is faithful: if $A_\eta^L=0$,
evaluation at $K_L$ gives zero integrals of all right-$K_L$-invariant
locally constant compactly supported functions against the inverse
pushforward of $\widetilde\eta$. Averaging arbitrary test functions
over $K_L$ then shows that this measure, and hence $\eta$, is zero.

For $D\in\DD_M$, the operator $T_D^L$ acts on $\SSS_M^L$.
Changing the level framing acts simultaneously on both endpoints
of its correspondence and preserves the density comparison.
Consequently $T_D^L$ commutes with $G_L$ under
\eqref{p1-schwartz}.

\begin{proposition}\label{prop:global-kernels}
For every $D\in\DD_M$ and every degree-$n$ division algebra $L$,
there is a unique $\mu_D^L\in\MM_E$ such that
\[
 T_D^L=I_L A_{\mu_D^L}^L I_L^{-1}
\]
under \eqref{p1-schwartz}.
\end{proposition}
\begin{proof}
Let $y_0=K_L$ and let $H_{D,0}$ be the fiber of $q_-$ over the
degree-zero framed bundle corresponding to $y_0$. If $H_{D,0}(F)$
is empty, the value at $y_0$ of $I_L^{-1}T_D^LI_L f$ is zero for
all $f$. Equivariance and transitivity imply $T_D^L=0$, so take
$\mu_D^L=0$.

Otherwise, write $d=\deg D$. Proposition~\ref{hecke-proper} and
Lemma~\ref{degree-divisibility} give $n\mid d$. Composing the
$q_+$-endpoint with $\psi^{-d/n}$ gives a map
$b_D:H_{D,0}(F)\to Y_E^L$ under \eqref{p1-space}. Transport the
half-density $dx_L^{1/2}$ at $b_D(h)$ to the $q_+$-endpoint of
$h\in H_{D,0}(F)$ using the same degree identification. Define
$\rho_D(h)$ by
\[
 (\alpha_{D,\vartheta})_h
 \bigl(dx_L^{1/2}|_{b_D(h)}\bigr)
 = \rho_D(h)\otimes dx_L^{1/2}|_{y_0}.
\]
Here $\rho_D(h)$ belongs to the density line of $T_hH_{D,0}(F)$,
which is the relative density line of $q_-$ at $h$.
This defines a positive smooth density $\rho_D$ on $H_{D,0}(F)$.
Consequently, with
$B_D=I_L^{-1}T_D^L I_L$ on $\SSS_E^L$, one has
\[
 (B_D f)(y_0)=\int_{H_{D,0}(F)} (f\circ b_D)\,\rho_D
             =\int_{Y_E^L} f\,d\lambda_D,\qquad
 \lambda_D=(b_D)_*\rho_D .
\]
Properness of $q_-$ gives compactness of $H_{D,0}(F)$, so
$\lambda_D$ is a compactly supported Radon measure.
Equivariance of $B_D$ and the fact that $K_L$ fixes $y_0$
imply that $\lambda_D$ is $K_L$-invariant.

Let $\widetilde\lambda_D$ be the unique right-$K_L$-invariant lift
of $\lambda_D$ to $G_L$. Its left $K_L$-invariance follows from that
of $\lambda_D$. Write $\inv(h)=h^{-1}$ and let
$\pi_L:G_L\to X_E^L\xrightarrow{p_E^L}\XX_E$ be the quotient map.
Set
\[
 \mu_D^L=(\pi_L)_*\inv_*\widetilde\lambda_D.
\]
Its bi-$K_L$-invariant lift is $\inv_*\widetilde\lambda_D$.
The inverse in the definition of $A_{\mu_D^L}^L$ therefore gives
$(A_{\mu_D^L}^L f)(y_0)=\int_{Y_E^L}f\,d\lambda_D$.
Both operators are $G_L$-equivariant, and $G_L$ acts transitively
on $Y_E^L$, so they agree everywhere. Faithfulness of the
kernel action gives uniqueness, including when the fiber is empty.
\end{proof}

\begin{remark}
The measure $\mu_D^L$ vanishes exactly when some $c\in|D|$ has
$n\nmid[k(c):F]$.
\end{remark} 
\subsection{The case of divisors of degree $n$}

Let $D=V(P(t))\subset\mathbb A^1_F$ be a closed separable point of
degree $n$ disjoint from $M$, with $P$ monic and irreducible. \begin{lemma}
\[
 \mu_D^L =c_{L,D}\,\delta_{P^-},\qquad
 P^-(u)=(-1)^nP(-u),\qquad
 c_{L,D}=\int_{H_{D,0}(F)}\rho_D>0,
\] where $\delta_{P^-}$ is the unit Dirac measure at $P^-$.
 \end{lemma}
\begin{proof} Put $C=\mathbb P^1_F$, $\mathcal K=\mathcal O_C(\infty)$, and $L_E=L\otimes_F E$. Let $s$ be the canonical section of $\mathcal K$ with divisor $\infty$. The fiber $H_{D,0}$ fixes the source of the modification to be $\mathcal E_0=\mathcal O_C\otimes_F L$ with the identity framing along $M$. An $F$-point of this fiber is represented by a right-$L$-linear inclusion $i:\mathcal E_0\hookrightarrow\mathcal E_1$ and a framing $\phi_1:\mathcal E_1|_M\xrightarrow{\sim}L_E$ such that $\phi_1\circ i|_M=\mathrm{id}$. The quotient is annihilated by the ideal of $D$ and is locally free of rank $n$ over $\mathcal O_D$. Its Morita transform over a splitting field has rank one on $D$.
\par
The reduced degree of $\mathcal E_1$ is $\deg D=n$. By the description of the degree components on $\mathbb P^1_F$, choose a right-$L$-linear isomorphism $\beta:\mathcal E_1\xrightarrow{\sim}\mathcal K\otimes_F L$. The composite $\beta i$ is determined by the global section $(\beta i)(1)$. Since $ts$ and $s$ form a basis of $H^0(C,\mathcal K)$, there are unique $A,B\in L$ such that
\[
 (\beta i)(x)=s(At+B)x.
\]
At $\infty$, $ts$ is a local generator of $\mathcal K$. In that generator the multiplier is $A+t^{-1}B$, whose value at $\infty$ is $A$. Since the quotient is supported on $D\subset\mathbb A^1_F$, the map at $\infty$ is invertible. Hence $A\in L^\times$.
\par
Let $\ell_c$ denote left multiplication by $c\in L^\times$ on $\mathcal K\otimes_F L$. Replace $\beta$ by the explicitly defined isomorphism $\beta'=\ell_{A^{-1}}\circ\beta$. In these coordinates the inclusion is
\[
 (\beta'i)(x)=s(t-z)x,\qquad z=-A^{-1}B.
\]
The order of multiplication follows from right $L$-linearity. Indeed any other initial choice is $\ell_c\circ\beta$, since a global endomorphism of $\mathcal K\otimes_F L$ is left multiplication by a global section of $\mathcal O_C\otimes_F L$, hence by an element of $L$. Its coefficients are $cA,cB$, and $-(cA)^{-1}(cB)=z$.
\par
Write $\chi_L(z)(u)=\operatorname{Nrd}(u-z)$. After splitting $L$ and applying Morita equivalence, the inclusion on the affine line is the matrix $tI-Z$, where $Z$ represents $z$. At each geometric point of $D$ its cokernel has length one, so the determinant has a simple zero there (by diagonal reduction over the local discrete valuation ring). Away from $D$ the map is invertible. Thus $\det(tI-Z)$ and $P(t)$ are monic of degree $n$ with the same roots, giving $\chi_L(z)=P$.
\par
Conversely, Proposition \ref{local-quotient} provides  $z\in L$ with $\chi_L(z)=P$. The section $s(t-z)$ defines an inclusion because it is generically invertible and its source is locally free. At $\infty$ its leading coefficient is $1$. Over a field splitting $L$ and $P$, the matrix $Z$ is diagonalizable with the distinct roots of $P$ as eigenvalues. The Morita transform of the cokernel is therefore the direct sum of one residue field at each root. Undoing Morita equivalence gives dimension $n$ at each root, with the ideal of $D$ acting by zero. These quotient conditions descend to $F$.
\par
To determine the framing, use $\beta'$ and $s|_M$ to write a section of $\mathcal E_1|_M$ as $s|_M\,y$, where $y\in L_E$. Since $t|_M=\tau$, the inclusion sends $x$ to $s|_M(\tau-z)x$. The element $\tau-z$ is invertible because its reduced norm is $P(\tau)\in E^\times$ by disjointness of $D$ and $M$. Thus the unique compatible framing is
\[
 \phi_1(s|_M\,y)=(\tau-z)^{-1}y.
\]
In particular the inclusion just constructed, equipped with this framing, gives an $F$-point of $H_{D,0}$.
\par
The degree identification $\psi^{-1}$ tensors the target with $\mathcal K^{-1}$ equipped with the inverse level trivialization. By the chosen level trivialization, $\nu(s|_M)=\tau^{-1}$, so the inverse trivialization sends $s^{-1}|_M$ to $\tau$. In degree zero the endpoint is consequently represented by
\[
 h_z=(\tau-z)^{-1}\tau=(1-\tau^{-1}z)^{-1}.
\]
Since $\tau$ is central and $\tau^2=a$, we have
\[
 h_z=(1+\tau^{-1}z)(1-a^{-1}z^2)^{-1}.
\]
Indeed $P$ is coprime to $t^2-a$, so $z^2-a$ is invertible in $F[z]\simeq F[t]/(P)$. The final factor lies in $L^\times$ and disappears in $L_E^\times/L^\times$. By the parameter convention, the double coset of $1+\tau^{-1}w$ has parameter $\chi_L(w)$. Hence the endpoint has parameter $P$.
\par
The kernel construction in the manuscript first pushes $\rho_D$ forward by the endpoint map, then lifts the resulting measure using Haar probability on the fibers, applies inversion, and finally passes to the double-coset parameter. Every lift of the endpoint lies in the same right coset as $h_z$; its inverse therefore lies in the double coset of $h_z^{-1}=1+\tau^{-1}(-z)$. Its parameter is
\[
 \chi_L(-z)(u)=\det(uI+Z)=(-1)^nP(-u)=P^-(u).
\]
Thus the measure is concentrated at $P^-$. Each pushforward preserves total mass, as does the lift because its fiber measures have mass one. The mass is therefore $\int_{H_{D,0}(F)}\rho_D$. Properness of the Hecke projection makes this fiber compact. It is nonempty by the construction above, and its positive smooth density gives positive mass to a coordinate neighborhood of any point. The integral is consequently finite and strictly positive.

\end{proof} \begin{proposition}\label{prop:degree-n-mass}
Let $F$ be the non-Archimedean local field of the manuscript. Let
$D=V(P(t))$ be a closed separable point of degree $n$ on the affine line,
disjoint from the level divisor $M$, with $P$ monic. Fix the cotangent
density $\vartheta_D$ used in the Hecke correspondence. Then the mass
$c_{L,D}$, and hence the measure $\mu_D^L$, is independent of the
central division $F$-algebra $L$ of degree $n$.
\end{proposition}
\begin{proof}
Write $A=F[t]/(P)$ and $\xi=t\bmod P$. We use the preceding description
of the Hecke fiber $H=H_{D,0}(F)$ over the identity framing $y_0$.
The mass to be compared is $\int_H\rho_D$.
\par\medskip\noindent
\textit{1. The fiber as a homogeneous space.}
The extension $A/F$ has degree $n$, so restriction to $A$ kills the
Brauer invariant of $L$. Thus $A$ splits $L$, and the maximal-subfield
criterion gives an embedding $A\hookrightarrow L$. Write $z$ for the
image of $\xi$. The modification $h_z$ is represented by $t-z$.
Every point of $H$ is uniquely of the form $h_{z'}$ with reduced
characteristic polynomial $\chi_L(z')=P$. Each such $z'$ gives an
embedding of $A$ in $L$. Skolem--Noether makes all these embeddings
conjugate. The centralizer of $z$ equals the centralizer of the maximal
subfield $A$, hence equals $A$. Consequently
\[
 \Phi_L:L^\times/A^\times\longrightarrow H,\qquad
 gA^\times\longmapsto h_{gzg^{-1}},
\]
is bijective.
\par
This analytic map has an analytic inverse. Indeed, composing with the
endpoint map $b_D$ and using its finite chart gives conjugation of $z$
up to the fixed sign of the chart. Differentiating conjugation at $1$
gives $X\mapsto Xz-zX$, whose kernel is $A$. Thus $d\Phi_L$ is injective
on $L/A$. The source has dimension $n^2-n$. The target is smooth of
the same dimension: after a separable splitting extension, the fixed
source bundle admits one independent elementary modification at each
of the $n$ points above $D$. Each choice is parametrized by a projective
space of dimension $n-1$. Hence the differential is invertible. The
analytic inverse function theorem gives a local inverse near $h_z$.
Translation by $g$ gives one near every $g h_z$, by equivariance.
On overlaps these inverses agree by injectivity.
\par\medskip\noindent
\textit{2. Two invariant densities on the fiber.}
For a nondegenerate bilinear form $B$ on a vector space $V$, let
$\lambda_B(v_1,\ldots,v_m)=|\det(B(v_i,v_j))|_F^{1/2}$ on each basis.
Changing basis by $S$ multiplies this value by $|\det S|_F$.
Use this construction for $B_L(x,y)=\Trd_{L/F}(xy)$ on $L$ and
$B_A(x,y)=\tr_{A/F}(xy)$ on $A$. Both are nondegenerate: over a
splitting field the first is the matrix trace pairing, and separability
gives nondegeneracy of the second. Denote the resulting densities by
$\lambda_L,\lambda_A$ and their left translates on the unit groups
by $dg_L,dg_A$.
\par
At the identity coset the quotient density $\lambda_Q$ on $L/A$ is
defined on a basis $\bar X$ by
\[
 \lambda_Q(\bar X)=\lambda_L(a,X)/\lambda_A(a),
\]
where $a$ is any basis of $A$ and $X$ is any lift of $\bar X$. A change
of $a$ cancels between numerator and denominator, and adding vectors
of $A$ to the lifts leaves the numerator unchanged. Conjugation by
$A^\times$ preserves $B_L$ and fixes $A$ pointwise. Thus $\lambda_Q$
extends by translation to the quotient density $d\dot g_L$.
\par
Let $\eta=dx_L^{1/2}$ be the invariant half-density on the space of
framings, using the period identification at the positive endpoint.
The defining equation for $\rho_D$ is
\[
 \alpha_{D,\vartheta,h}(\eta_{b_D(h)})
       =\rho_D(h)\otimes\eta_{y_0}.
\]
For $g\in L^\times$, changing the framing fixes $y_0$, sends $b_D(h)$
to $b_D(gh)$, and commutes with $\alpha_{D,\vartheta}$. Applying $g$
to the equation and cancelling $\eta_{y_0}$ gives
$\rho_D(gh)=g_*\rho_D(h)$. The ratio of the two positive densities is
therefore constant on the orbit. Put
\[
 b_{L,D}=\frac{(\Phi_L^*\rho_D)_{A^\times}}{\lambda_Q}.
 \qquad\text{Then}\qquad
 c_{L,D}=b_{L,D}\,\vol(L^\times/A^\times,d\dot g_L).
\]
We compare the two factors separately.
\par\medskip\noindent
\textit{3. Comparison at the identity coset.}
We use the algebraic determinant-line comparison underlying the
manuscript's critical-density proposition. Its functoriality under
isomorphisms and scalar extension is essential here.
Choose a nonzero $\ell$ in the $F$-line $\Nm_{A/F}(K_C|_D)$, where
$K_C$ is the cotangent bundle of the curve. First use $|\ell|$ as the
cotangent density. Normalize $dx_L$ at $y_0$ by $\lambda_L$, using
$T_{y_0}Y_E^L=L_E/L\simeq L$ via the $\tau^{-1}$ coefficient.
At the other endpoint use its translate. This is well-defined since
the stabilizer preserves the trace pairing. Multiplying $dx_L$ by
a constant scales both endpoint half-densities equally.
\par
For clarity, write $K(V)=\det V^*$ and let
$d(B)\in K(V)^{\otimes2}$ be the determinant tensor of $B$.
On any basis it evaluates to the Gram determinant. Thus
$\lambda_B^2=|d(B)|$ and $(\lambda_B^{1/2})^4=|d(B)|$.
Put $Q=L/A$. The determinant sequence for $A\subset L$ gives
\[
 d_Q=d(B_L)/d(B_A)\in K(Q)^{\otimes2},
 \qquad \lambda_Q^4=|d_Q^{\otimes2}|.
\]
Here and below division of tensors means contraction with the inverse
of a nonzero tensor in a one-dimensional line.
\par
At $h_z$ take the fourth tensor power of the canonical-line comparison,
using $\ell$ in the cotangent norm factor. Feed into it the determinant
tensor of the trace pairing transported to the positive endpoint.
Divide the output by the corresponding tensor at $y_0$. This leaves
a tensor in $K(T_{h_z}H)^{\otimes4}$: the endpoint half-densities have
become squared canonical lines, and the fiber density has become a
fourth canonical power. Pull this tensor back by $d\Phi_L$ and call it
$R_L\in K(Q)^{\otimes4}$. Taking absolute values gives exactly
$(\Phi_L^*\rho_D)^4$ at the identity coset for the normalization
$|\ell|$. Thus the fourth power of the desired ratio is
\[
 |s_L|_F,\qquad s_L=R_L/d_Q^{\otimes2}\in F^\times .
\]
All operations defining $s_L$ are on algebraic lines over $F$.
\par
Now choose a common finite separable extension $S/F$ splitting $P$,
$L$, and a second division algebra $L'$ of degree $n$. In matrix
realizations over $S$, the images of $\xi$ have the same $n$ distinct
eigenvalues. Matching their eigenspaces yields an algebra isomorphism
$L_S\simeq L'_S$ carrying $z$ to $z'$, hence identifying the two
embeddings of $A_S$ element by element. It carries $t-z$ to $t-z'$,
preserves the level coordinate and the period twist, and identifies
the two orbit differentials. It preserves both trace forms.
Functoriality of the canonical-line comparison therefore carries
$R_L$ to $R_{L'}$ and $d_Q$ to its counterpart. Their ratios agree
in $S$, so $s_L=s_{L'}$ in $F$. If $\vartheta_D=a_D|\ell|$, the factor
$\mathcal T_D^{-(n-1)/2}$ in the density comparison gives the same
multiplier $a_D^{(n-1)/2}$ for each fiber density upon contraction.
Positivity now gives $b_{L,D}=b_{L',D}$ for the prescribed $\vartheta_D$.
\par\medskip\noindent
\textit{4. Volume of the orbit.}
Write $q=|k_F|$, $q_A=|k_A|$, and $e=e(A/F)$. Use additive Haar
measure on $F$ giving $\OO_F$ volume one. The trace determinant on an
integral basis of $\OO_A$ generates the discriminant ideal
$\mathfrak d_{A/F}$. The proportion of units in $\OO_A$ is
$1-q_A^{-1}$. On the units, additive trace density agrees with its
multiplicative translate, since left multiplication by a unit preserves
the integral lattice. Hence
\[
 \vol(\OO_A^\times,dg_A)=|\mathfrak d_{A/F}|_F^{1/2}(1-q_A^{-1}).
\]
For $L$, use the cyclic presentation with $W/F$ unramified of degree
$n$, $\Pi^n=\varpi_F$, and $\Pi x\Pi^{-1}=\sigma(x)$ on $W$, for a
generator $\sigma$ of $\operatorname{Gal}(W/F)$. The maximal order
is $\bigoplus_{i=0}^{n-1}\OO_W\Pi^i$. The trace pairing on the
index-zero summand is unimodular. For $i>0$, its only nonzero block
in row $i$ is in column $n-i$, and on coefficients it is
\[
 (x,y)\longmapsto\varpi_F\tr_{W/F}(x\sigma^i(y)).
\]
After permuting block columns there are $n-1$ blocks of size $n$
of this kind. Each contributes $n$ to the determinant valuation.
The additive volume of $\OO_L$ is therefore $q^{-n(n-1)/2}$.
Since $\OO_L/\Pi\OO_L=k_W$ has $q^n$ elements, the unit volume is
$q^{-n(n-1)/2}(1-q^{-n})$.
\par
Finally normalize $v_L$ by $v_L(\Pi)=1$. Its restriction to $F^\times$
is $n v_F$, so its restriction to $A^\times$ is $(n/e)v_A$, with
$v_A$ normalized. Consequently the quotient has $n/e$ valuation
classes. Each is a translate of the image of $\OO_L^\times$.
That image has volume $\vol(\OO_L^\times)/\vol(\OO_A^\times)$:
in the quotient integration formula for the indicator of $\OO_L^\times$,
the integral along each coset meeting $\OO_L^\times$ is precisely
$\vol(\OO_A^\times)$. We obtain
\[
 \vol(L^\times/A^\times,d\dot g_L)
 =\frac ne\,\frac{q^{-n(n-1)/2}(1-q^{-n})}
 {|\mathfrak d_{A/F}|_F^{1/2}(1-q_A^{-1})}.
\]
This depends only on $A/F$ and $n$. The comparison in step 3 now gives
independence of $c_{L,D}$. Finally the preceding Dirac-measure lemma
gives $\mu_D^L=c_{L,D}\delta_{P^-}$, where $P^-(u)=(-1)^nP(-u)$,
so the measure is independent of $L$ as well.
\end{proof}

\begin{conjecture}\label{conj:global-kernels}
For each $D\in\DD_M$, the measure $\mu_D^L$ on $\widehat\Xi_E$
from Proposition~\ref{prop:global-kernels} is independent of the
degree-$n$ division algebra $L$.
\end{conjecture}

\section*{Acknowledgements}
During the preparation of this paper, I was partially supported by
the ERC grant no.~101142781. I thank AI for locating imprecisions and
helping to insert corrections.

\end{document}